\documentclass[11pt]{amsart}
\usepackage[utf8]{inputenc}
\usepackage{amsmath}
\usepackage{amsfonts}
\usepackage{amssymb}
\usepackage{amsthm}
\usepackage{graphicx}
\usepackage{enumerate}
\usepackage{subcaption}
\usepackage[colorlinks=true,linkcolor=blue,citecolor=blue]{hyperref}
\usepackage{tikz}

\newtheorem{theorem}{Theorem}[section]
\newtheorem{corollary}[theorem]{Corollary}
\newtheorem{lemma}[theorem]{Lemma}

\newtheorem{fact}[theorem]{Fact} 
\theoremstyle{definition}
\newtheorem{definition}[theorem]{Definition}

\newtheorem{remark}[theorem]{Remark}

\numberwithin{equation}{section}

\newcommand{\Fix}{\operatorname{Fix}}
\newcommand{\wto}{\stackrel{w.}{\rightharpoonup}}
\renewcommand{\H}{\mathcal{H}}

\newcommand{\qede}{\hspace*{\fill}$\Diamond$\medskip}
\DeclareMathOperator*{\argmin}{arg\,min}

\title[Cyclic Douglas--Rachford Method]{The Cyclic Douglas--Rachford Method for Inconsistent Feasibility Problems}

\author[J.M. Borwein]{Jonathan M. Borwein}
\address[J.M. Borwein]{CARMA Centre, University of Newcastle, Callaghan, NSW 2308, Australia.}
\email{{\tt jon.borwein@gmail.com}}

\author[M.K. Tam]{Matthew K. Tam}
\address[M.K. Tam]{CARMA Centre, University of Newcastle, Callaghan, NSW 2308, Australia.}
\email{{\tt matthew.k.tam@gmail.com}}

\date{\today}

\begin{document}

\begin{abstract}
We analyse the behaviour of the newly introduced cyclic Douglas--Rachford algorithm for finding a point in the intersection of a finite number of closed convex sets. This work considers the case in which the target intersection set is possibly empty.
\end{abstract}

\maketitle

\section{Preliminaries and Notation}
Throughout we assume $\H$ is a (real) \emph{Hilbert space} with inner product $\langle \cdot,\cdot\rangle$ and induced norm $\|\cdot\|$. We use $\to$ (resp. $\wto$) to denote norm (resp. weak) convergence.

We consider the \emph{convex feasibility problem}
 \begin{equation}\label{eq:cfp}
  \text{Find }x\in\bigcap_{i=1}^NC_i,
 \end{equation}
where $C_i$ are closed convex subsets of $\H$. For convenience, we define $C_0:=C_{N}$ and $C_{N+1}:=C_1$.

When the intersection is empty, (\ref{eq:cfp}) is ill-posed.  Instead, we seek an appropriate substitute for a point in the intersection. For example, if $N=2$ it is natural to consider the following variational problem
 \begin{equation}\label{eq:varprob}
  \inf_{(c_1,c_2)\in C_1\times C_2}\|c_1-c_2\|.
 \end{equation}
If the infimum is realised, we call the solution $(c_1,c_2)$ a \emph{best approximation pair} with respect to $(C_1,C_2)$. For $N>2$, an appropriate generalization of (\ref{eq:varprob}) for characterizing the limit cycles of projection methods remains elusive and subtle. For details, see \cite{BCC12}.

\subsection{General Theory}
In this section we recall some general theory regarding general nonexpansive mappings  ---
applied later to projections and reflections --- in Hilbert space.  We give some definitions.

 \begin{definition} Let $T:\H\to\H$.

 \begin{enumerate}[(a)]
  \item $T$ is \emph{nonexpansive} if
   $\|Tx-Ty\|\leq\|x-y\|\text{ for all }x,y\in\H.$
  \item $T$ is \emph{firmly nonexpansive} if
   $$\|Tx-Ty\|^2+\|(I-T)x-(I-T)y\|^2\leq \|x-y\|^2\text{ for all }x,y\in\H.$$
  \item $T$ is \emph{$\lambda$-averaged} if
   $T=(1-\lambda)I+\lambda R,$
  for some nonexpansive mapping $R:\H\to\H$.
  \item $T$ is \emph{demiclosed} if
   $x_n\wto x, Tx_n\to y \implies Tx=y.$
  \item $T$ is \emph{asymptotically regular at $x\in\H$} if
   $(I-T)T^nx\to 0.$
  \item The set of \emph{fixed points} of $T$ is
   $\Fix T:=\{x\in\H:Tx=x\}.$
 \end{enumerate}\end{definition}

We collect facts concerning the interplay between these properties.

\begin{fact}[Nonexpansive properties]\label{fact:nonexpansive}
 The following hold.
 \begin{enumerate}[(a)]
  \item If $T$ is firmly nonexpansive then $T$ is nonexpansive.
  \item If $\lambda\in[0,1]$ and $T$ is $\lambda$-averaged then $T$ is nonexpansive.
  \item $T$ is firmly nonexpansive if and only if $T$ is $1/2$-averaged.
  \item If $T$ is nonexpansive then $I-T$ is demiclosed.
  \item If $\lambda\in]0,1[$ and $T$ is $\lambda$-averaged then, for any $z\in\H$,
   $$\lim_{n\to\infty}\|T^{n+1}z-T^nz\|=\frac{1}{k}\lim_{n\to\infty}\|T^{n+k}z-T^nz\|=\lim_{n\to\infty}\frac{1}{n}\|T^nz\|,$$
 for all $k\geq 1$. In particular, if $(T^nz)_{n=1}^\infty$ is bounded, then $T$ is asymptotically regular at $x$.
 \end{enumerate}
\end{fact}
\begin{proof}
 For (a)--(d) see, for example, \cite[Ch.~4]{BC11}. For (e), see \cite[Th.~2.1]{BBR78}.
\end{proof}

The following theorem will be useful is establishing convergence of our algorithms.

\begin{theorem}[Weak convergence of iterates]\label{th:averagedcvgt}
 Let $\lambda\in]0,1[$. Suppose $(T_i)_{i=1}^m$ is a family of $\lambda$-averaged mappings from $\H$ to $\H$ such that $\Fix(T_m\dots T_1)\neq\emptyset$. For any $x_0\in\H$ define
  $$x_{n+1}:=(T_m\dots T_1)x_n.$$
 Then $x_n-(T_m\dots T_1)x_n\to 0$ and there exists points $y_1\in\Fix(T_m\dots T_1),\,y_2\in\Fix (T_1T_m\dots T_2),\,\dots,\,y_m\in\Fix(T_{m-1}\dots T_1T_m)$ such that
  \begin{align*}
   x_n&\wto y_1 =T_my_m,\\
   T_1x_n&\wto y_2=T_1y_1, \\
   T_2T_1x_n&\wto y_{3}=T_2y_2,\\
    &\;\;\vdots  \\
   T_{m-2}\dots T_1x_n&\wto y_{m-1}=T_{m-2}y_{m-2},\\
   T_{m-1}T_{m-2}\dots T_1x_n&\wto y_m=T_{m-1}y_{m-1}.
  \end{align*}
\end{theorem}
\begin{proof}
 This is a special case of \cite[Th.~5.22]{BC11}.
\end{proof}

\subsection{The Method of Cyclic Projections}

The \emph{(nearest point) projection} onto a set $C$ (if it exists) is the mapping $P_C:\H\to C$ defined by
 $$P_Cx:=\argmin_{c\in C}\|c-x\|.$$
It is well known that if $C$ is closed and convex, $P_C$ is well defined (i.e., nearest points exist uniquely for all $x\in\H$) (see, for example, \cite[Prop.~2.1.2]{BBL97}). It has the variational characterization
 $$P_Cx=c \iff c\in C\text{ and }\langle x-P_Cx,C-P_Cx\rangle\leq0\text{ for all }c\in C.$$

For any $y_0\in\H$, the \emph{method of cyclic projections} can be described in terms of the  iteration scheme
 \begin{equation*}
  y_1^1:=P_{C_1}y_0,\quad y^{i+1}_n:=P_{C_{i+1}}y_n^i,\quad y^1_{n+1}:=y^{N+1}_n.
 \end{equation*}
We refer to the sequences $(y_n^1)_{n=1}^\infty,(y_n^2)_{n=1}^\infty,\dots,(y_n^N)_{n=1}^\infty$ as the \emph{cyclic projection sequences}.

Define
 \begin{equation*}
  Q_i:=P_{C_{i}}P_{C_{i-1}}\dots P_{C_1}P_{C_N}\dots P_{C_{i+1}}.
 \end{equation*}
Note that, for each $i$, the sequence $(y_n^i)_{n=1}^\infty$ is given by
 \begin{equation*}
  y^{i}_{n+1}=Q_iy^{i}_n.
 \end{equation*}

Suppose that each $\Fix Q_i$ is nonempty and let $q^1\in\Fix Q_1$. Define the sequence $(q^i)_{i=1}^N$ by
  \begin{equation*}
   q^{i+1}:=P_{C_{i+1}}q^i\in\Fix Q_{i+1}.
  \end{equation*}
Define $(d^i)_{i=1}^N$, the sequence of \emph{difference vectors}, by $d^i:=q^{i+1}-q^i$. It can be shown that the difference vectors are well-defined (i.e., they are independent of the choice of $q^1$). For further details see \cite{BBL97}.

Recall the following dichotomy theorem.
\begin{theorem}[Cyclic projections dichotomy] \label{th:cycprojdich}
 Exactly one of the following alternatives hold.
 \begin{enumerate}[(a)]
  \item Each $\Fix Q_i$ is empty. Then $\|y_n^i\|\to +\infty$, for all $i$.
  \item Each $\Fix Q_i$ is nonempty. Then, for each $i$,  $(y_n^i)_{n=1}^\infty$ weakly converges to a point $y^i$ such that $y^{i+1}=P_{C_{i+1}}y^i$, and the sequence $(y_n^{i+1}-y_n^{i})_{n=1}^\infty$ converges in norm to $d^i$.
 \end{enumerate}
\end{theorem}
\begin{proof}
 See \cite[Th.~5.2.1]{BBL97}.
\end{proof}

\subsection{The Cyclic Douglas--Rachford Method}
The \emph{(metric) reflection} with respect to a set $C$ is the mapping $R_C:\H\to H$ given by
 $$R_C:=2P_C-I,$$
where $I$ denotes the \emph{identity mapping}. If $C$ is closed and convex, $R_C$ is well defined. It has the variational characterization (see, for example, \cite[Fac.~2.1]{BT13})
 $$R_Cx=r\iff \frac{1}{2}(r+x)\in C \text{ and }\langle x-r,c-r\rangle\leq\frac{1}{2}\|x-r\|^2\text{ for all }c\in C.$$

The \emph{Douglas--Rachford operator} is the mapping $T_{C_1,C_2}:\H\to\H$ given by
 $$T_{C_1,C_2}:=\frac{I+R_{C_2}R_{C_1}}{2}.$$

The \emph{cyclic Douglas--Rachford operator} is the mapping $T_{[C_1,C_2,\dots,C_N]}:\H\to\H$ given by
 $$T_{[C_1,C_2,\dots,C_N]}:=\prod_{i=1}^NT_{C_i,C_{i+1}},$$
where $T_{N,N+1}:=T_{N,1}$.

Where there is no ambiguity, we write $T_{i,i+1}$ to mean $T_{C_i,C_{i+1}}$, and $\sigma_i$ to mean the cyclic permutation of $C_1,C_2,\dots,C_N$ beginning with $C_i$. Under this notation
  $$T_{[\sigma_1]}:=T_{[C_1,C_2,\dots,C_{N-1},C_N]},\quad T_{[\sigma_2]}:=T_{[C_2,C_3,\dots,C_N,C_1]},\quad \text{etc}.$$
For convenience, we define $\sigma_0:=\sigma_N$ and $\sigma_{N+1}:=\sigma_1$.

For any $x_0\in\H$, the \emph{cyclic Douglas--Rachford method} can be described in terms of the iteration scheme
 \begin{equation}
  x^1_1:=x_0,\quad x_n^{i+1}:=T_{i,i+1}x^i_n,\quad x_{n+1}^1:=x_n^{N+1}.
 \end{equation}
We refer to the sequences $(x_n^1)_{n=1}^\infty,(x_n^2)_{n=1}^\infty,\dots,(x_n^N)_{n=1}^\infty$ as the \emph{cyclic Douglas--Rachford sequences}.

Note that, for each $i$, the sequence $(x_n^i)_{n=1}^\infty$ is given by
 \begin{equation}
  x_{n+1}^{i}:=T_{[\sigma_{i}]}x_n^{i}.
 \end{equation}

\begin{remark}\label{rem:conincide}
If $z\in C_i$ then $T_{i,i+1}z=P_{C_{i+1}}z$. Hence, if $x_0=y_0\in C_1$, the cyclic projection  and cyclic Douglas--Rachford sequences coincide. That is, for each $i$,
 $$y_n^i=x_n^i,\text{for }n=1,2,3,\dots.$$
If $x_0\neq y_0$ and $x_0\not\in C_1$, it is entirely possible for the cyclic projection and cyclic Douglas--Rachford sequences to be distinct. For an example, see \cite[Rem.~3.3]{BT13}. \qede
\end{remark}

\begin{remark}[Alternating reflections] The classical Douglas-Rachford method, which applies to two sets problems, performs iterations by repeated application of a Douglas-Rachford operator, i.e. $x_{n+1}:=T_{1,2}(x_n)$ for all $n$ or $x_{n+1}:=T_{2,1}(x_n)$ for all $n$.
Thus, in the two sets case, the cyclic Douglas--Rachford method may be thought of as a traditional Douglas-Rachford algorithm in which the set chosen to be reflected on first is alternated.\qede\end{remark}

\section{A Dichotomy Theorem}\label{sec:dichotomy}

We require a suite of seven preparatory lemmas.

\begin{lemma}\label{lem:averaged}
 For each index $i$,
 \begin{enumerate}[(a)]
  \item $T_{i,i+1}$ is $1/2$-averaged, and hence firmly nonexpansive.
  \item $T_{[\sigma_i]}$ is $(1-2^{-N})$-averaged.
 \end{enumerate}
\end{lemma}
\begin{proof}
 (a) Since convex reflections are nonexpansive, it immediately follows that $T_{i,i+1}$ is $1/2$-averaged. (b) Suppose that $T$ is $1/2$-averaged and $Q$ is $(1-2^{-k})$-averaged for some nonnegative integer $k$. We may write
  $$T=\frac{1}{2}I+\frac{1}{2}R,\text{ and }Q=\frac{1}{2^k}I+\left(1-\frac{1}{2^k}\right)S,$$
for nonexpansive mappings $R$ and $S$.
 Observe
 \begin{align*}
  TQ & =\frac{1}{2}Q+\frac{1}{2}RQ=\frac{1}{2^{k+1}}I+\frac{2^k-1}{2^{k+1}}S+\frac{2^k}{2^{k+1}}RQ\\
     & =\frac{1}{2^{k+1}}I+\frac{2^{k+1}-1}{2^{k+1}}\left(\frac{2^k-1}{2^{k+1}-1}S+\frac{2^k}{2^{k+1}-1}RQ\right) \\
     & =\frac{1}{2^{k+1}}I+\left(1-\frac{1}{2^{k+1}}\right)\left(\frac{2^k-1}{2^{k+1}-1}S+\frac{2^k}{2^{k+1}-1}RQ\right).
 \end{align*}
 Since $S$ and $RQ$ are nonexpansive, so is their convex combination, and hence $TQ$ is $(1-2^{k+1})$-averaged. The equivalence now follows.
\end{proof}

The follow lemma shows that the cyclic Douglas--Rachford method has similar asymptotic behaviour to the method of cyclic projections. To exploit the nonexpansive properties of $T_{[\sigma_i]}$ and $T_{i,i+1}$, we will sometimes choose $y_0:=P_{C_1}x_0$.

\begin{lemma}\label{lem:asymdiff}
 For any $x_0\in\H$, choose $y_0:=P_{C_1}x_0$. As $n\to\infty$,
 \begin{equation*}
  (x_n^i-x_n^{i+1})-(y_n^i-y_n^{i+1})\to 0,
 \end{equation*}
 for any index $i$.
\end{lemma}
\begin{proof}
 By Remark~\ref{rem:conincide}, the method cyclic projection sequence starting at $y_0$ can be consider as cyclic Douglas--Rachford sequence. Since $T_{[\sigma_1]}$ is nonexpansive,
  \begin{equation*}
   \|x^1_{n+1}-y^1_{n+1}\|\leq\|x^1_{n}-y^1_n\| \implies \lim_{n\to\infty} \|x^1_{n}-y^1_n\|\text{ exists}.
  \end{equation*}
 Since $T_{i,i+1}$ is firmly nonexpansive, for each $i$,
  \begin{align*}
   \|x^1_{n}-y^1_n\|^2-\|x^1_{n+1}-y^1_{n+1}\|^2
   &= \sum_{i=1}^N\left(\|x^i_{n}-y^i_{n}\|^2-\|x^{i+1}_{n}-y^{i+1}_{n}\|^2 \right) \\
   &\geq \sum_{i=1}^N \|(x_n^i-x_n^{i+1})-(y_n^i-y_n^{i+1})\|^2.
  \end{align*}
 The result follows by taking the limit as $n\to\infty$.
\end{proof}

\begin{lemma}\label{lem:fixedpoints}
 The sequence $(x^j_n)_{n=1}^\infty$ is bounded if and only if $\Fix T_{[\sigma_j]}$ is nonempty.
\end{lemma}
\begin{proof}
 Suppose $(x^j_n)_{n=1}^\infty$ is bounded. Then there exists a subsequence $(x^j_{n_k})_{k=1}^\infty$ weakly convergent to some point $z$. By Fact~\ref{fact:nonexpansive}(e), $(I-T_{[\sigma_j]})x_{n_k}\to 0$. Since $(I-T_{[\sigma_j]})$ is demiclosed, $(I-T_{[\sigma_j]})z=0\implies z\in\Fix T_{[\sigma_j]}$.

 Conversely, if $z\in\Fix T_{[\sigma_j]}$ nonexpansivity implies
  $$\|z-x^j_n\|\leq \|z-x^j_1\|\implies \|x^j_n\|\leq\|z\|+\|z-x^j_1\|.$$
 This completes the proof.
\end{proof}

\begin{lemma}\label{lem:bounded}
The following four properties are equivalent.
\begin{enumerate}[(a)]
  \item The sequence $(x^j_n)_{n=1}^\infty$ is bounded, for some index $j$.
  \item The sequences $(x^1_n)_{n=1}^\infty,(x^2_n)_{n=1}^\infty,\dots,(x^N_n)_{n=1}^\infty$ are bounded.
  \item The sequence $(y^j_n)_{n=1}^\infty$ is bounded, for some index $j$.
  \item The sequences $(y^1_n)_{n=1}^\infty,(y^2_n)_{n=1}^\infty,\dots,(y^N_n)_{n=1}^\infty$ are bounded.
\end{enumerate}
Furthermore, if $(x^j_n)_{n=1}^\infty$ is unbounded then $\|x^j_n\|\to+\infty$.
\end{lemma}
\begin{proof}
 Fix an index $j$. For any $x_0\in\mathcal H$, choose $y_0:=P_{C_1}x_0$. Since $T_{[\sigma_j]}$ is nonexpansive,
  $$\|x^j_n-y^j_n\|\leq \|x^j_{n-1}-y^j_{n-1}\|\leq \dots\leq \|x^j_1-y^j_1\|.$$
 By the triangle inequality
 $$\|x^j_n\|\leq \|x^j_1-y^j_1\|+\|y_n^j\|\text{ and }\|y_n^j\|\leq \|x_n^j\|+\|x^j_1-y^j_1\|.$$
 Thus $(x^j_n)_{n=1}^\infty$ is bounded if and only if $(y^j_n)_{n=1}^\infty$ is bounded. By Theorem~\ref{th:cycprojdich}, $(y^j_n)_{n=1}^\infty$ is bounded if and only if $(y^1_n)_{n=1}^\infty,(y^2_n)_{n=1}^\infty,\dots,(y^N_n)_{n=1}^\infty$, and if $(y^j_n)_{n=1}^\infty$ is unbounded then $\|x^j_x\|\to+\infty$. The result follows by combining these two statements.
\end{proof}

We observe
\begin{align}
 x^{i+1}_n:=T_{i,i+1}x^i_n
 &= \frac{x^i_n+R_{C_{i+1}}R_{C_i}x^i_n}{2} \notag \\
 &= \frac{x^i_n+2P_{C_{i+1}}R_{C_i}x^i_n-R_{C_i}x^i_n}{2} \notag \\
 &= \frac{2x^i_n+2P_{C_{i+1}}R_{C_i}x^i_n-2P_{C_i}x^i_n}{2} \notag \\
 &= x^i_n+P_{C_{i+1}}R_{C_i}x^i_n-P_{C_i}x^i_n. \notag \\
 \implies P_{C_{i+1}}R_{C_i}x^i_n &= x^{i+1}_n-x^i_n+P_{C_i}x^i_n \label{eq:TABxn}
\end{align}

\begin{lemma}\label{lem:nbound}
 For all $i$ and for all $n$,
  \begin{equation}\label{eq:nbound1}
   \|x^{i+1}_n-P_{C_{i+1}}x^{i+1}_n\|^2\leq \langle x^{i+1}_n-P_{C_{i+1}}x^{i+1}_n,x^i_n-P_{C_i}x^i_n\rangle.
  \end{equation}
 In particular, for all $i$ and all $n$,
 \begin{equation}\label{eq:nbound2}
  \|x^{i+1}_n-P_{C_{i+1}}x^{i+1}_n\|\leq \|x^i_n-P_{C_i}x^i_n\|.
 \end{equation}
\end{lemma}
\begin{proof}
By \eqref{eq:TABxn} and the variational characterization of convex projections,
\begin{align*}
 &\|x^{i+1}_n-P_{C_{i+1}}x^{i+1}_n\|^2-\langle x^{i+1}_n-P_{C_{i+1}}x^{i+1}_n,x^i_n-P_{C_i}x^i_n\rangle \\
 &=\langle x^{i+1}_n-P_{C_{i+1}}x^{i+1}_n, (x^{i+1}_n-x^i_n+P_{C_i}x^i_n)-P_{C_{i+1}}x^{i+1}_n\rangle \\
 &=\langle x^{i+1}_n-P_{C_{i+1}}x^{i+1}_n, P_{C_{i+1}}R_{C_i}x^i_n-P_{C_{i+1}}x^{i+1}_n\rangle \\
 &\leq 0.
\end{align*}
This proves \eqref{eq:nbound1}. Equation \eqref{eq:nbound2} now follows by an application of the Cauchy--Schwarz inequality.
\end{proof}

\begin{lemma}\label{lem:sum2}
  For all $i$ and for all $m$,
  \begin{multline}\label{eq:doublesum}
     \sum_{n=2}^m\sum_{i=1}^N \|(x^{i+1}_n-P_{C_{i+1}}x^{i+1}_n)-(x^{i}_n-P_{C_{i}}x^{i}_n)\|^2 \\
   \leq \langle x^{1}_2-P_{C_{1}}x^{1}_2,x^N_1-P_{C_{N}}x^N_1\rangle-\langle x^{1}_{m+1}-P_{C_{1}}x^{1}_{m+1},x^N_m-P_{C_N}x^N_m\rangle,
  \end{multline}
  where $\langle x^{1}_2-P_{C_{1}}x^{1}_2,x^N_1-P_{C_{N}}x^N_1\rangle$ and $\langle x^{1}_{m+1}-P_{C_{1}}x^{1}_{m+1},x^N_m-P_{C_N}x^N_m\rangle$ are nonnegative. In particular, the double-sum in \eqref{eq:doublesum} is bounded, and hence, as $n\to\infty$,
   $$(x^{i+1}_n-P_{C_{i+1}}x^{i+1}_n)-(x^{i}_n-P_{C_{i}}x^{i}_n)\to 0.$$
\end{lemma}
\begin{proof}
Applying Lemma~\ref{lem:nbound},
\begin{align*}
 &\sum_{n=2}^m\sum_{i=1}^N \|(x^{i+1}_n-P_{C_{i+1}}x^{i+1}_n)-(x^{i}_n-P_{C_{i}}x^{i}_n)\|^2  \\
 &= \sum_{n=2}^m\sum_{i=1}^N\left(\|x^{i+1}_n-P_{C_{i+1}}x^{i+1}_n\|^2-2\langle x^{i+1}_n-P_{C_{i+1}}x^{i+1}_n,x^i_n-P_{C_i}x^i_n\rangle\right. \\
 &\qquad \left. +\|x^{i}_n-P_{C_{i}}x^{i}_n\|^2\right) \\
 &\leq\sum_{n=2}^m\sum_{i=1}^N\left(\langle x^{i}_n-P_{C_{i}}x^{i}_n,x^{i-1}_n-P_{C_{i-1}}x^{i-1}_n\rangle-\langle x^{i+1}_n-P_{C_{i+1}}x^{i+1}_n,x^i_n-P_{C_i}x^i_n\rangle\right). \\
 &=\langle x^{1}_2-P_{C_{1}}x^{1}_2,x^{0}_2-P_{C_{0}}x^{0}_2\rangle-\langle x^{N+1}_m-P_{C_{N+1}}x^{N+1}_m,x^N_m-P_{C_N}x^N_m\rangle \\
 &=\langle x^{1}_2-P_{C_{1}}x^{1}_2,x^N_1-P_{C_{N}}x^N_1\rangle-\langle x^{1}_{m+1}-P_{C_{1}}x^{1}_{m+1},x^N_m-P_{C_N}x^N_m\rangle.
\end{align*}
The nonnegativity of $\langle x^{1}_2-P_{C_{1}}x^{1}_2,x^N_1-P_{C_{N}}x^N_1\rangle$ and $\langle x^{1}_{m+1}-P_{C_{1}}x^{1}_{m+1},x^N_m-P_{C_N}x^N_m\rangle$ is a consequence of \eqref{eq:nbound1}.
\end{proof}

We now prove the analogue of Lemma~\ref{lem:nbound} which will be applied to the limits (if they exist) of the cyclic Douglas--Rachford sequences. As before, we may deduce
\begin{equation} P_{C_{i+1}}R_{C_i}x^i = x^{i+1}-x^i+P_{C_i}x^i. \label{eq:TABx}
\end{equation}

\begin{lemma}\label{lem:limitpoints}
 Suppose $(x^i)_{i=1}^N$ are points such that $x^{i+1}=T_{i,i+1}x^i$. Then, for all $i$,
  \begin{equation*}
   P_{C_{i+1}}R_{C_i}x^i-P_{C_{i+1}}x^{i+1}=(x^{i+1}-x^i)-(P_{C_{i+1}}x^{i+1}-P_{C_i}x^i)=0.
  \end{equation*}
\end{lemma}
\begin{proof}
 Consider the cyclic Douglas--Rachford sequences for initial point $x_0:=x^1$. Since $x^{i}\in\Fix T_{[\sigma_i]}$, for each $i$, the result follows from \eqref{eq:TABx} and Lemma~\ref{lem:sum2}.
\end{proof}

We are now ready to prove a dichotomy theorem which is the cyclic Douglas--Rachford method analogue of Theorem~\ref{th:cycprojdich}.

\begin{theorem}[Cyclic Douglas--Rachford dichotomy]\label{th:dichotmy}
The following holds.
\begin{enumerate}[(a)]
 \item For each $i$,
  $$P_{C_{i+1}}R_{C_i}x_n^i-P_{C_{i+1}}x_n^i=(x_n^{i+1}-x_n^i)-(P_{C_{i+1}}x_n^{i+1}-P_{C_i}x_n^i)\to 0.$$
 \item Exactly one of the following alternatives hold.
  \begin{enumerate}[(i)]
   \item Each $\Fix T_{[\sigma_i]}$ is empty. Then $\|x_n^i\|\to+\infty$, for all $i$.
   \item Each $\Fix T_{[\sigma_i]}$ is nonempty. Then, for each $i$,
   $$x_n^i\wto x^{i}\in\Fix T_{[\sigma_i]} \text{ with }x^{i+1}=T_{i,i+1}x^{i}.$$
   Furthermore, for each $i$,
    \begin{align*}
      &x_n^{i+1}-x_n^i=P_{C_{i+1}}R_{C_i}x_n^i-P_{C_i}x_n^i\to d^i,&& P_{C_{i+1}}x_n^{i+1}-P_{C_i}x_n^i\to d^i,\\
    &x^{i+1}-x^i=P_{C_{i+1}}x^{i+1}-P_{C_i}x^i=d^i,&& P_{C_{i+1}}R_{C_i}x^i=P_{C_{i+1}}x^{i+1}.
    \end{align*}
  \end{enumerate}
  \end{enumerate}
\end{theorem}
\begin{proof}
 (a) follows by Lemma~\ref{lem:sum2}. (b) By appealing to Lemmas~\ref{lem:fixedpoints} and \ref{lem:bounded}, we establish the two possible alternatives: either $\Fix T_{[\sigma_i]}=\emptyset$ and $\|x^i_n\|\to+\infty$ for all $i$, or $\Fix T_{[\sigma_i]}\neq\emptyset$ for all $i$.

 If each $\Fix T_{[\sigma_i]}\neq\emptyset$,  Lemma~\ref{lem:averaged} together with Theorem~\ref{th:averagedcvgt} imply that the sequence $(x_n^i)_{n=1}^\infty$ converges weakly to a point $x^i\in\Fix T_{[\sigma_i]}$ with $x^{i+1}=T_{i,i+1}x^i$.  Lemma~\ref{lem:asymdiff} with Theorem~\ref{th:cycprojdich} implies $x_n^{i+1}-x_n^i\to d^i$, which together with (a) implies $P_{C_{i+1}}x_n^{i+1}-P_{C_i}x_n^i\to d^i$.

Lemma~\ref{lem:limitpoints} together with Theorem~\ref{th:cycprojdich} applied to the cyclic Douglas--Rachford sequences having initial point $x^1$ yields $x^{i+1}-x^i=P_{C_{i+1}}x^{i+1}-P_{C_i}x^i=d^i$ and  $P_{C_{i+1}}R_{C_i}x_n^i=P_{C_{i+1}}x_n^{i+1}$.
\end{proof}

If $\bigcap_{i=1}^NC_i \neq \emptyset$, it can be shown that the limits $(x^i)_{i=1}^N$ coincide (see, for example, \cite[Lem.~2.3]{BT13}). In this case, we obtain \cite[Th.~3.1]{BT13} as a special case of Theorem~\ref{th:dichotmy}. That is, we have the following corollary.

\begin{corollary}[Consistent cyclic Douglas--Rachford iterations]\label{cor:cycDR}
 Suppose $\cap_{i=1}^NC_i\neq\emptyset$. Then the cyclic Douglas--Rachford sequences weakly converge to a common point $x$ such that $P_{C_i}x=P_{C_j}x$ for all indices $i,j$. In particular, $P_{C_j}\in\cap_{i=1}^NC_i$ for any index $j$.
\end{corollary}

\begin{remark}
The proof of Corollary~\ref{cor:cycDR} given in \cite{BT13} for the consistent case is dependent on the fact that $\Fix T_{[\sigma_1]}=\cap_{i=1}^N\Fix T_{i,i+1}\neq\emptyset$. Since $\Fix T_{i,i+1}\neq\emptyset$ if and only if $C_i\cap C_{i+1}\neq\emptyset$, in the inconsistent case one can only  guaranteed that $\Fix T_{[\sigma_1]}\supseteq\cap_{i=1}^N\Fix T_{i,i+1}=\emptyset$. \qede
\end{remark}

\begin{remark}[Approximating the difference vectors]
In Theorem~\ref{th:dichotmy}, it was shown that the sequences
 $$(x_n^{i+1}-x_n^i)_{n=1}^\infty,\quad (P_{C_{i+1}}R_{C_i}x_n^i-P_{C_i}x_n^i)_{n=1}^\infty,\quad
(P_{C_{i+1}}x_n^{i+1}-P_{C_i}x_n^i)_{n=1}^\infty,$$
converge (in norm) to $d^i$. The latter two are suitable if one is interested in approximating $d^i$ using a pair of points from $C_i$ and $C_{i+1}$.\qede
\end{remark}

\begin{remark}[Cyclic Douglas--Rachford as a favourable compromise]
 The behaviour of the cyclic Douglas--Rachford scheme is somewhere between that of the method of alternating projections and the classical Douglas--Rachford scheme. In this sense, it can be consider a comprise between the two schemes having some of the desirable properties of both. We elaborate.

Firstly, the cyclic Douglas--Rachford and classical Douglas--Rachford scheme perform the reflections with respect to the constraints sets, rather than using just a projection, as is the case of the method of cyclic projections. This can be seen as an advantage (at least heuristically). If a point is not contained in a particular constraint set, the reflection can potentially yield a strictly feasibility problem, where as projections produces point on the boundary (see Figure~\ref{fig:bdry}).

\begin{figure}[h]
 \begin{center}
 \begin{tikzpicture}[scale=0.7]
   \filldraw [green!30] plot [smooth cycle] coordinates {(0,0) (2,0.5) (3,3) (0,4) (-2,2)};
   \fill (-4,0) circle (0.05) node[below] {$x$};
   \draw (-1.6,1.3) -- (-4,0);
   \fill[->] (-1.6,1.3) circle (0.05) node[above]  {$P_Cx$};
   \fill[->] (0.8,2.6) circle (0.05) node[above] {$R_Cx$};
   \draw[dotted] (-1.6,1.3) -- (0.8,2.6);
   \draw (1,1) node {$C$};
 \end{tikzpicture}
 \end{center}
 \caption{$R_Cx$ is strictly feasible, while $P_Cx$ is on the boundary of $C$.}\label{fig:bdry}
\end{figure}
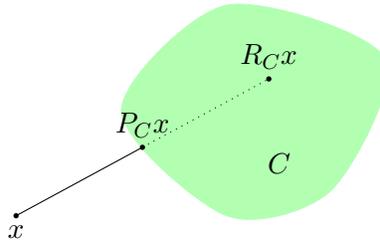

On one hand, the cyclic Douglas--Rachford and classical Douglas--Rachford iterations both proceed by applying a two set Douglas--Rachford mapping (i.e. one of the form $T_{i,i+1}$). In the consistent case, the limit obtained by both these schemes, once projected onto an appropriate constraint set, produces a solution to a feasibility problem. On the other, in the inconsistent case, the Douglas--Rachford scheme 
iterates are always unbounded. The behaviour described in Theorem~\ref{th:dichotmy} is much closer to that of the method of cyclic projections, described in Theorem~\ref{th:cycprojdich}.

 Thus,  if one wishes to diagnose infeasibility one might prefer Douglas-Rachford to the cyclic variant, but if one desires  an estimate even in the infeasible case one would likely opt for the cyclic variant. The behaviour of the three methods is illustrated in the two possible two sets cases in Figure~\ref{fig:dich}.~\qede
\end{remark}

\section{The Two Set Case}
We now specialize the results of Section~\ref{sec:dichotomy} for the case of problems having only two sets. Here the geometry of the problem is both better understood and more tractable. We introduce the following two sets
 $$E:=\{e\in C_1:d(e,C_1)=d(C_1,C_2)\},\quad F:=\{f\in C_2:d(f,C_1)=d(C_1,C_2)\}.$$
Further, the \emph{displacement vector}, $v$, is defined by
 $$v:=P_{\overline{C_2-C_1}}(0)=-P_{\overline{C_1-C_2}}(0).$$
We recall some useful facts.
\begin{fact}[Properties of $E,F$ and $v$]\label{fact:twosets}
 The following hold.
 \begin{enumerate}[(i)]
  \item If $C_1\cap C_2\neq\emptyset$ then $C_1\cap C_2=E=F$ and $v=0$.
  \item $E+v=F$, and $\|v\|=d(C_1,C_2)$ if and only if $C_2-C_1$ is closed.
  \item $E=\Fix Q_1=\Fix(P_{C_1}P_{C_2})$ and $F=\Fix Q_2=\Fix(P_{C_2}P_{C_1})$.
  \item $v=d^1=-d^2$.
 \end{enumerate}
\end{fact}
\begin{proof}
 See, for example, \cite[Sec.~1]{BB93} and \cite[Fac.~5.2.2]{BBL97}.
\end{proof}

\begin{figure}[h]
\begin{subfigure}[b]{0.3\textwidth}
\begin{center}
\begin{tikzpicture}[scale=1.25]
\draw [green!50, thick] (0.3,0) --(3,0);
\fill [blue!20] (0.3,2.95) -- plot[domain=0.5:3] (\x,{1/\x+0.1}) -- (3,2.95) -- cycle;
\draw (0.7,-0.2) node {$C_2$};
\draw (0.7,2.6) node {$C_1$};
\draw[->,black!50] (1.0000,0.5000)--(1.2533,0.8979);
\draw[->,black!50] (1.2533,0.8979)--(1.2533,0.0000);
\draw[->,black!50] (1.2533,0.0000)--(1.5586,0.7416);
\draw[->,black!50] (1.5586,0.7416)--(1.5586,0.0000);
\draw[->,black!50] (1.5586,0.0000)--(1.7706,0.6648);
\draw[->,black!50] (1.7706,0.6648)--(1.7706,0.0000);
\draw[->,black!50] (1.7706,0.0000)--(1.9353,0.6167);
\draw[->,black!50] (1.9353,0.6167)--(1.9353,0.0000);
\draw[->,black!50] (1.9353,0.0000)--(2.0712,0.5828);
\draw[->,black!50] (2.0712,0.5828)--(2.0712,0.0000);
\draw[->,black!50] (2.0712,0.0000)--(2.1876,0.5571);
\draw[->,black!50] (2.1876,0.5571)--(2.1876,0.0000);
\draw[->,black!50] (2.1876,0.0000)--(2.2899,0.5367);
\draw[->,black!50] (2.2899,0.5367)--(2.2899,0.0000);
\draw[->,black!50] (2.2899,0.0000)--(2.3816,0.5199);
\draw[->,black!50] (2.3816,0.5199)--(2.3816,0.0000);
\draw[->,black!50] (2.3816,0.0000)--(2.4648,0.5057);
\draw[->,black!50] (2.4648,0.5057)--(2.4648,0.0000);
\draw[->,black!50] (2.4648,0.0000)--(2.5412,0.4935);
\draw[->,black!50] (2.5412,0.4935)--(2.5412,0.0000);
\draw[->,densely dotted,black!50] (2.6,0.25)-- (3,0.25);
\draw[fill] (1.0000,0.5000) circle (0.02);
\draw (-0,-1.11) node {};
\end{tikzpicture}
\caption{MAP}
\end{center}
\end{subfigure}
\begin{subfigure}[b]{0.3\textwidth}
\begin{center}
\begin{tikzpicture}[scale=1.25]
\draw [green!50, thick] (0.3,0) --(3,0);
\fill [blue!20] (0.3,2.95) -- plot[domain=0.5:3] (\x,{1/\x+0.1}) -- (3,2.95) -- cycle;
\draw (0.7,-0.2) node {$C_2$};
\draw (0.7,2.6) node {$C_1$};
\draw[fill] (1.0000,0.5000) circle (0.02);
\draw[->,black!50] (1.0000,0.5000)--(1.2533,-0.3979);
\draw[->,densely dotted,black!50] (1.2533,-0.3979)--(1.6554,-1.1020);
\draw (-0,-1.11) node {};
\end{tikzpicture}
\caption{DR}
\end{center}
\end{subfigure}
\begin{subfigure}[b]{0.3\textwidth}
\begin{center}
\begin{tikzpicture}[scale=1.25]
\draw [green!50, thick] (0.3,0) --(3,0);
\fill [blue!20] (0.3,2.95) -- plot[domain=0.5:3] (\x,{1/\x+0.1}) -- (3,2.95) -- cycle;
\draw (0.7,-0.2) node {$C_2$};
\draw (0.7,2.6) node {$C_1$};
\draw[fill] (1.0000,0.5000) circle (0.02);
\draw[->,black!50] (1.0000,0.5000)--(1.2533,-0.3979);
\draw[->,black!50] (1.2533,-0.3979)--(1.4430,0.3951);
\draw[->,black!50] (1.4430,0.3951)--(1.5788,-0.3383);
\draw[->,black!50] (1.5788,-0.3383)--(1.6999,0.3500);
\draw[->,black!50] (1.6999,0.3500)--(1.7952,-0.3071);
\draw[->,black!50] (1.7952,-0.3071)--(1.8860,0.3231);
\draw[->,black!50] (1.8860,0.3231)--(1.9606,-0.2869);
\draw[->,black!50] (1.9606,-0.2869)--(2.0343,0.3047);
\draw[->,black!50] (2.0343,0.3047)--(2.0963,-0.2724);
\draw[->,black!50] (2.0963,-0.2724)--(2.1587,0.2909);
\draw[->,black!50] (2.1587,0.2909)--(2.2121,-0.2612);
\draw[->,black!50] (2.2121,-0.2612)--(2.2666,0.2800);
\draw[->,black!50] (2.2666,0.2800)--(2.3137,-0.2522);
\draw[->,black!50] (2.3137,-0.2522)--(2.3623,0.2711);
\draw[->,black!50] (2.3623,0.2711)--(2.4046,-0.2448);
\draw[->,black!50] (2.4046,-0.2448)--(2.4486,0.2637);
\draw[->,black!50] (2.4486,0.2637)--(2.4871,-0.2384);
\draw[->,black!50] (2.4871,-0.2384)--(2.5274,0.2572);
\draw[->,black!50] (2.5274,0.2572)--(2.5628,-0.2329);
\draw[->,black!50] (2.5628,-0.2329)--(2.6001,0.2517);
\draw (-0,-1.11) node {};
\end{tikzpicture}
\caption{Cyclic DR}
\end{center}
\end{subfigure}

\vspace{1ex}

\begin{subfigure}[b]{0.3\textwidth}
\begin{center}
\begin{tikzpicture}[scale=1.25]
\draw [green!50,thick] (-1.5,0) -- (1.5,0);
\fill [blue!20] (-1,2.55) -- plot[domain=-1.5:1.5] (\x,{\x*\x+0.3}) -- (1.5,2.55) -- cycle;
\draw (-1,-0.2) node {$C_2$};
\draw (-1,2.1) node {$C_1$};
\draw[->,black!50] (1.0000,0.5000)--(0.6689,0.7475);
\draw[->,black!50] (0.6689,0.7475)--(0.6689,0.0000);
\draw[->,black!50] (0.6689,0.0000)--(0.3598,0.4295);
\draw[->,black!50] (0.3598,0.4295)--(0.3598,0.0000);
\draw[->,black!50] (0.3598,0.0000)--(0.2128,0.3453);
\draw[->,black!50] (0.2128,0.3453)--(0.2128,0.0000);
\draw[->,black!50] (0.2128,0.0000)--(0.1303,0.3170);
\draw[->,densely dotted,black!50] (0.1303,0.3170)--(0.1303,0.0000);
\draw[->,densely dotted,black!50] (0.1303,0.0000)--(0.0808,0.3065);
\draw[->,densely dotted,black!50] (0.0808,0.3065)--(0.0808,0.0000);
\draw[->,densely dotted,black!50] (0.0808,0.0000)--(0.0503,0.3025);
\draw[->,densely dotted,black!50] (0.0503,0.3025)--(0.0503,0.0000);
\draw[->,densely dotted,black!50] (0.0503,0.0000)--(0.0314,0.3010);
\draw[->,densely dotted,black!50] (0.0314,0.3010)--(0.0314,0.0000);
\draw[->,densely dotted,black!50](0.0314,0.0000)--(0.0196,0.3004);
\draw[fill] (1.0000,0.5000) circle (0.02);
\draw[fill] (0,0.3) circle (0.02);
\draw[fill] (0,0) circle (0.02);
\draw (-0,-1) node {};
\end{tikzpicture}
\caption{MAP}
\end{center}
\end{subfigure}
\begin{subfigure}[b]{0.3\textwidth}
\begin{center}
\begin{tikzpicture}[scale=1.25]
\draw [green!50,thick] (-1.5,0) -- (1.5,0);
\fill [blue!20] (-1,2.55) -- plot[domain=-1.5:1.5] (\x,{\x*\x+0.3}) -- (1.5,2.55) -- cycle;
\draw (-1,-0.2) node {$C_2$};
\draw (-1,2.1) node {$C_1$};
\draw[->,black!50] (1.0000,0.5000)--(0.6689,-0.2475);
\draw[->,densely dotted,black!50] (0.6689,-0.2475)--(0.2948,-0.6344);
\draw[->,densely dotted,black!50] (0.2948,-0.6344)--(0.1020,-0.9448);
\draw (-0,-1) node {};
\end{tikzpicture}
\caption{DR}
\end{center}
\end{subfigure}
\begin{subfigure}[b]{0.3\textwidth}
\begin{center}
\begin{tikzpicture}[scale=1.25]
\draw [green!50,thick] (-1.5,0) -- (1.5,0);
\fill [blue!20] (-1,2.55) -- plot[domain=-1.5:1.5] (\x,{\x*\x+0.3}) -- (1.5,2.55) -- cycle;
\draw (-1,-0.2) node {$C_2$};
\draw (-1,2.1) node {$C_1$};
\draw[->,black!50] (1.0000,0.5000)--(0.6689,-0.2475);
\draw[->,black!50] (0.6689,-0.2475)--(0.4454,0.2509);
\draw[->,black!50] (0.4454,0.2509)--(0.3363,-0.1622);
\draw[->,black!50] (0.3363,-0.1622)--(0.2415,0.1962);
\draw[->,black!50] (0.2415,0.1962)--(0.1889,-0.1395);
\draw[->,black!50] (0.1889,-0.1395)--(0.1389,0.1798);
\draw[->,black!50] (0.1389,0.1798)--(0.1098,-0.1323);
\draw[->,densely dotted,black!50] (0.1098,-0.1323)--(0.0814,0.1744);
\draw[->,densely dotted,black!50] (0.0814,0.1744)--(0.0647,-0.1298);
\draw[->,densely dotted,black!50] (0.0647,-0.1298)--(0.0481,0.1725);
\draw[->,densely dotted,black!50] (0.0481,0.1725)--(0.0382,-0.1290);
\draw[->,densely dotted,black!50] (0.0382,-0.1290)--(0.0284,0.1718);
\draw[->,densely dotted,black!50] (0.0284,0.1718)--(0.0226,-0.1287);
\draw[->,densely dotted,black!50] (0.0226,-0.1287)--(0.0168,0.1716);
\draw[->,densely dotted,black!50] (0.0168,0.1716)--(0.0134,-0.1286);
\draw[->,densely dotted,black!50] (0.0134,-0.1286)--(0.0100,0.1715);
\draw[fill] (1.0000,0.5000) circle (0.02);
\draw[fill] (0.0047,-0.1285) circle (0.02);
\draw[fill] (0.0035,0.1715) circle (0.02);
\draw (-0,-1) node {};
\end{tikzpicture}
\caption{Cyclic DR}
\end{center}
\end{subfigure}
 \caption{Behaviour of the three methods starting with the same initial point. In (A)--(F), $C_2:=\mathbb R\times\{0\}$. In (A)--(C), $C_1:=\operatorname{epi}(1+1/\cdot)\cap(\mathbb R_+\times\mathbb R_+)$ and $E,F$ are empty. In (D)--(F), $C_1:=\operatorname{epi}(1+(\cdot)^2)$ and $E,F$ are nonempty.
 }\label{fig:dich}
\end{figure}
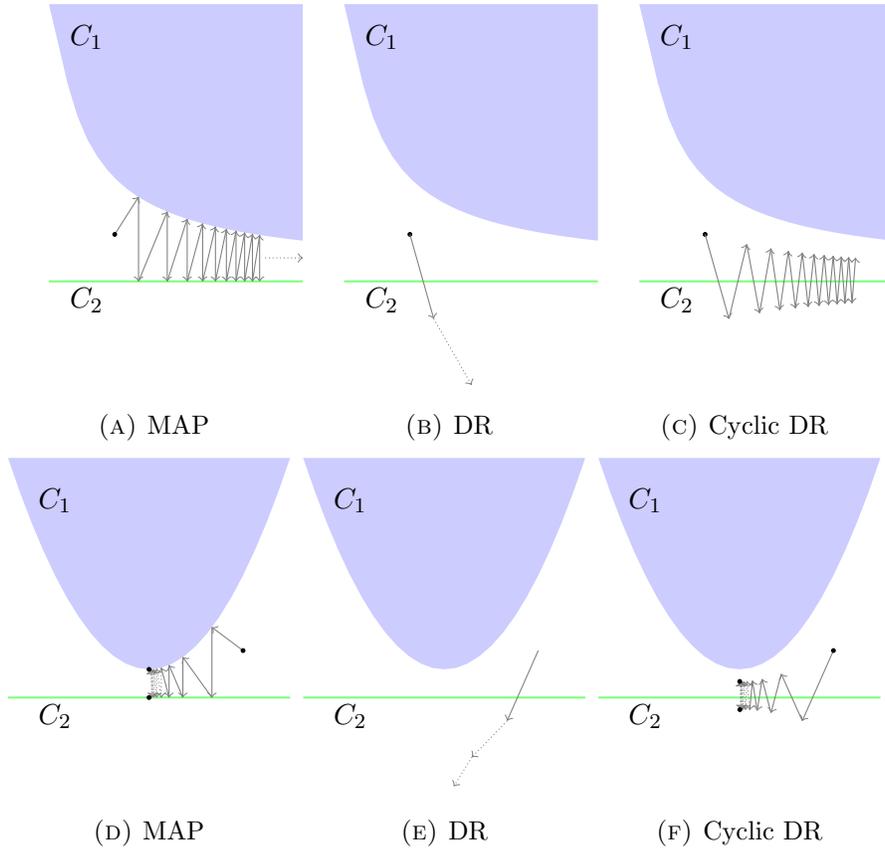

We are now ready to specialize the conclusions of Theorem~\ref{th:dichotmy}. In particular, we show that the cyclic Douglas--Rachford scheme can can be used to find best approximation pairs, provided they exist.
\begin{theorem}[Alternating Douglas--Rachford dichotomy]\label{th:twosetdichotomy}
The following holds.
\begin{enumerate}[(a)]
 \item We have
  \begin{align*}
   P_{C_{2}}R_{C_1}x_n^1-P_{C_{2}}x_n^1=(x_n^{2}-x_n^1)-(P_{C_{2}}x_n^{2}-P_{C_1}x_n^1) &\to 0,\\
  P_{C_{1}}R_{C_2}x_n^2-P_{C_{1}}x_n^2=(x_n^{1}-x_n^2)-(P_{C_{1}}x_n^{1}-P_{C_2}x_n^2)&\to 0.
  \end{align*}
 \item Exactly one of the following alternatives hold.
  \begin{enumerate}[(i)]
   \item $E,F,\Fix T_{[C_1,C_2]},\Fix T_{[C_2,C_1]}=\emptyset$. Then $\|x_n^1\|,\|x_n^2\|\to+\infty$.
   \item $E,F,\Fix T_{[C_1,C_2]},\Fix T_{[C_2,C_1]}\neq\emptyset$. Then
   $$x_n^1\wto x^{1}\in\Fix T_{[C_1,C_2]},\quad x_n^2\wto x^{2}\in\Fix T_{[C_2,C_1]},$$
   with $x^{2}=T_{C_1,C_2}x^{1}$ and $x^{1}=T_{C_2,C_1}x^{2}.$
   Furthermore, 
    \begin{align*}
      &x_n^{2}-x_n^1=P_{C_{2}}R_{C_1}x_n^1-P_{C_1}x_n^1\to v,&& P_{C_{2}}x_n^{2}-P_{C_1}x_n^1\to v,\\
      &x_{n+1}^1-x_n^2=P_{C_{1}}R_{C_2}x_n^2-P_{C_2}x_n^2\to -v,&& P_{C_{1}}x_{n+1}^{1}-P_{C_2}x_n^2\to -v,
    \end{align*}
    and $x^{2}-x^1=P_{C_{2}}x^{2}-P_{C_1}x^1=v$. In particular, 
    \begin{align*}
    & P_{C_{1}}R_{C_2}x^2=P_{C_{1}}x^{1}\in E,&&P_{C_{2}}R_{C_1}x^1=P_{C_{2}}x^{2}\in F.
    \end{align*}
  \end{enumerate}
  \end{enumerate}
\end{theorem}
\begin{proof}
 Follows from Theorem~\ref{th:dichotmy} and Fact~\ref{fact:twosets}.
\end{proof}

Contrast Theorem~\ref{th:twosetdichotomy} with its analogues for  cyclic projections (Theorem~\ref{th:cycprojdich}) and for the classical Douglas--Rachford scheme (Theorem~\ref{th:drdichotomy}), which we state below for completeness.

\begin{theorem}[Douglas--Rachford method dichotomy]\label{th:drdichotomy}
Let $C_1,C_2\subseteq \H$ be closed and convex. Let $z_0\in\mathcal H$ and set $z_{n+1}:=T_{C_1,C_2}z_n$. Then
 \begin{enumerate}[(a)]
  \item $z_{n+1}-z_n=P_{C_2}R_{C_1}z_n-P_{C_1}z_n\to v$ and $P_{C_2}P_{C_1}z_n-P_{C_1}z_n\to v$.
  \item  Exactly one of the following alternatives holds.
   \begin{enumerate}[(i)]
    \item $C_1\cap C_2\neq\emptyset$ and $(z_n)_{n=1}^\infty$ converges weakly to a point in $$\Fix T_{C_1,C_2}=(C_1\cap C_2)+N_{\overline{C_1-C_2}}(0).$$
    \item $C_1\cap C_2=\emptyset$ and $\|z_n\|\to+\infty$.
   \end{enumerate}
  \item Exactly one of the following two alternatives holds.
  \begin{enumerate}[(i)]
   \item $E,F=\emptyset$, $\|P_{C_1}z_n\|\to+\infty$ and $\|P_{C_2}P_{C_1}z_n\|\to+\infty$;
   \item $E,F\neq\emptyset$, $(P_{C_1}z_n)_{n=1}^\infty$ and $(P_{C_2}P_{C_1}z_n)_{n=1}^\infty$ are bounded with weak cluster points in $E$ and $F$, respectively. Furthermore, the weak cluster points of
     $$((P_{C_1}z_n,P_{C_2}R_{C_1}z_n))_{n=1}^\infty \text{ and }  ((P_{C_1}z_n,P_{C_2}P_{C_1}z_n))_{n=1}^\infty$$
   are best approximation pairs relative to $(C_1,C_2)$.
  \end{enumerate}
\end{enumerate}
\end{theorem}

\begin{proof}
 See \cite[Th.~3.13]{BCL04}.
\end{proof}

\section{Final Remarks}
We have analysed the behaviour of the cyclic Douglas--Rachford algorithm for finding a point in the intersection of a finite number of closed convex sets. Whilst each iteration of the the method is similar to that of the classical Douglas--Rachford scheme, its behaviour, particularly in the inconsistent case,  is closer the that of the method of cyclic projections. With this in mind, one might consider the cyclic Douglas--Rachford scheme as a useful comprise between these methods.

Applied to two-set feasibility problems for which best approximation pairs exist, the cyclic Douglas--Rachford method produces a pair of points which when projected onto the appropriate sets yields a best approximation pair. This is important for applications in which consistency of the feasibility problem is not known \emph{a priori}.

Finally we finish with two open questions:
\begin{enumerate}
\item Can one prove a version of the main result in \cite{BCC12},  showing that for cyclic Douglas-Rachford applied to three sets there is no variational characterization of the fixed point sets Fix $T_{[\sigma_i]}$ being nonempty?
    \item What can be said about convergence rates for the Douglas-Rachford methods? Much less seems known than in the case of alternating projections. Recent linear convergence results for the classical Douglas--Rachford method applied to affine subspaces can be found in \cite{BCPW13}.
    \end{enumerate}

\subsection*{Acknowledgements} The authors would like to thank Liangjin Yao for his helpful suggestions.


\footnotesize
\begin{thebibliography}{99}


\bibitem{BBR78}
J.B. Baillon, R.E. Bruck and S. Reich.
\newblock On the asymptotic behavior of nonexpansive mappings and semigroups in Banach spaces.
\newblock Houst. J. Math., 4(1):1--9 (1978).

\bibitem{BCC12}
J.B. Baillon, P.L. Combettes and R. Cominetti.
\newblock There is no variational characterization of the cycles in the method of periodic projections.
\newblock J. Func. Anal., 262(1):400--408 (2012).

\bibitem{BB93}
H.H. Bauschke and J.M. Borwein.
\newblock On the convergence of von Neumann's alternating projection algorithm for two sets.
\newblock Set-Valued Anal., 1:185--212 (1993).

\bibitem{BBL97}
H.H. Bauschke, J.M. Borwein and A.S. Lewis.
\newblock The method of cyclic projections for closed convex sets in Hilbert spaces.
\newblock Contemp. Math., 204:1--38 (1997).

\bibitem{BC11}
H.H. Bauschke and P.L. Combettes.
\newblock Convex analysis and monotone operator theory in Hilbert spaces.
\newblock Springer (2011).

\bibitem{BCL04}
H.H. Bauschke, P.L. Combettes and D.R. Luke.
\newblock Finding best approximation pairs relative to two closed convex sets in Hilbert spaces.
\newblock J. Approx. Theory, 127:178--192 (2004).

\bibitem{BCPW13}
H.H. Bauschke,  J.Y. Bello Cruz, T.T.A. Nghia, H.M. Phan and X. Wang.
\newblock The rate of linear convergence of the Douglas--Rachford algorithm for subspaces is the cosine of the Friedrichs angle. 
\newblock J. Approx. Theory, in press (2014).
\newblock doi: \href{http://linkinghub.elsevier.com/retrieve/pii/S0021904514001166}{10.1016/j.jat.2014.06.002}.

\bibitem{BT13}
J.M. Borwein and M.K. Tam.
\newblock A cyclic Douglas--Rachford iteration scheme.
\newblock J. Optim. Theory Appl., 160:1--29 (2014).
\newblock doi: \href{http://link.springer.com/article/10.1007%2Fs10957-013-0381-x}{10.1007/s10957-013-0381-x}.

\end{thebibliography}
\end{document}